\documentclass{amsart}

\usepackage[T1]{fontenc}

\usepackage{amsmath, amssymb,amsthm}
\usepackage{thmtools}
\usepackage{array}
\usepackage{graphicx}
\usepackage{float}
\usepackage{caption,subcaption}
\usepackage{tikz}
\usepackage{tikz-cd}
\usetikzlibrary{matrix,arrows,decorations.pathmorphing}
\usepackage[mathscr]{euscript}
\usepackage{eufrak}
\usepackage{MnSymbol}

\def\R{\mathbb{R}}
\def\C{\mathbb{C}}

\def\N{\mathbb{N}}
\def\P{\mathbb{P}}

\def\O{\mathcal{O}}
\def\X{\mathfrak{X}}
\def\Y{\mathfrak{Y}}
\def\ZZ{\mathfrak{Z}}
\def\B{\mathfrak{B}}

\def\L{\mathfrak{L}}
\def\U{\mathscr{U}}

\def\M{\mathscr{M}}

\def\V{\mathscr{V}}
\def\f{\mathfrak{f}}

\def\CC{\mathscr{C}}

\newtheorem{thm}{Theorem}[section]

\newtheorem{rem}[thm]{Remark}

\newtheorem{cor}[thm]{Corollary}
\newtheorem{prop}[thm]{Proposition}

\title{Gap for geometric canonical height functions}
\author{YUGANG ZHANG}
\date{\today}

\begin{document}
\maketitle

\begin{abstract}
We prove the existence of a gap around zero for canonical height functions associated with endomorphisms of projective spaces defined over complex function fields. We also prove that if the rational points of height zero are Zariski dense, then the endomorphism is birationally isotrivial. As a corollary, by a result of S. Cantat and J. Xie, we have a geometric Northcott property on projective plane in the same spirit of results of R. Benedetto, M. Baker and L. Demarco  on the projective line.

\end{abstract}

\section{Introduction}
Let $K:=\C(\B)$ be the function field of a smooth complex projective variety $\B $ of dimension $b\geq 1$. Let $f$ be an endomorphism of $\P^k_K$ defined over $K$ and of degree $d \geq 2$, i.e. such that $f^*\O(1) \cong \O(d)$. We can view it as a holomorphic family of endomorphism $f_t$ of $\P^k_\C,$ parameterized by $t$ in a Zariski open subset $\Lambda$ of $\B$. In this article, we study the dynamics of such an endomorphism defined over a function field, i.e. the behavior of orbits $\{f^m(x) \}_{m\in \N}$ for $K-$rational points $x\in \P^k_K$.

We say that $f$ is {\em isotrivial} (resp. {\em birationally isotrivial}) if for any two parameters $t,t^{'}\in \Lambda$, the endomophisms $f_t,f_{t^{'}} $ of $\P^k_\C$ are conjugated by an automorphism (resp. a birational self map) of $\P^k_\C$. These maps are somehow the trivial ones, their dynamics do not change with the parameter.

The main basic tool to study the dynamics of the map $f$ is its canonical height $\hat{h}_f: \P^k(\overline{K}) \to \R_+$. We refer to Section \ref{section height function} for the precise definition of the canonical height function. Roughly speaking, it measures how complicated the orbit of a point $x\in \P^k_K(\overline{K})$ is under $f$. For example, preperiodic points of $f$ will be of height zero (but the converse may not be true, see \cite{Baker09}). We say that a geometric point $x \in \P^k_K(\overline{K})$ is {\em stable} if its canonical height is zero. 

The following is our main theorem.

\begin{thm}
\label{theorem}
    Let $f: \P^k_K \to \P^k_K$ be an endomorphism of degree $d\geq 2$ defined over the field $K$ of rational functions of a smooth complex projective variety. 
    \begin{enumerate}
        \item There exists a positive real number $\varepsilon > 0$ (depending on $f$ and $K$) such that, for a $K$-point $x\in \P^k_K$, $\hat{h}_f(x)< \varepsilon$ implies $\hat{h}_f(x)=0.$ 
        \item If the stable $K-$points are Zariski dense, then $f$ is birationally isotrivial.
    \end{enumerate}
\end{thm}
\begin{rem}\normalfont
    The results hold true for polarized endomorphisms, which are defined as follows. Let $X$ be a smooth projective variety defined over $K$, a polarized endomorphism of $X$ of degree $d\geq 2$ defined over $K$ is an endomorphism of $X$ defined over $K$ such that there exists an ample line bundle $L$ defined over $K$ such that $f^*L \cong dL$ (We use additive notation for line bundles). For the first point of Theorem \ref{theorem}, we can even ask $X$ to be only normal since the arguments are purely algebraic.
\end{rem}
We say that there is a {\em gap} of the canonical height function around zero. Note that the point (2) is already known by a result of T. Gauthier and G. Vigny (\cite{GV19Pn}), but we propse here a simpler proof which does not rely on complex pluripotential theory. In particular, we do not use bifurcation currents in the course of our proof.

We now focus on the case of the projective plane $\P^2_K$. By a result of S. Cantat and J. Xie (\cite{cantat2020birational}), in this case, birational isotriviality of $f$ is equivalent to isotriviality of $f$. Given $C$ a singular curve of genus zero, and $g$ an endomorphism of $C$, we say that $g$ is {\em isotrivial} if the normalization of $g$ is so. If $C$ is of genus one, then we say that $C$ is {\em isotrivial} if there exists a base change $\B' \to \B$ and a complex elliptic curve $C'$ such that the normalization of $C_{K'}:= C\times_K K'$ is isomorphic to $C'\times_\C K'$, where $K'=\C(\B')$. If $E$ is a (possibly reducible) proper subvariety of $\P^2_K$, we say that it is {\em preperiodicly isotrivial} if every (irreducible) curve $C$ of $E$ has an iteration $f^m(C)$ which is periodic of period $l$ by $f$ and $f^l: f^m(C) \to f^m(C)$ is isotrivial. Moreover, if there is a zero dimensional component $x$ of $E$, then by definition, an iteration by $f$ of $x$ should belong to a curve of $E$.

We have the following corollary in dimension two.

\begin{cor}
\label{corollary}
Let $f: \P^2_K \to \P^2_K$ be an endomorphism of degree $d\geq 2$ defined over $K$ and suppose that $f$ is not isotrivial. Then there exists $\varepsilon>0$ and a preperiodicly isotrivial proper subvariety $E$ (which can be empty or reducible) so that the set
$$\{x\in \P_K^2\setminus E(K) \ |\ \hat{h}_f(x) < \varepsilon \}$$
is finite. In particular, there are only finitely many preperiodic $K$-points in $X\setminus E$.
\end{cor}

We now give some motivations and historical notes for this work, see also the introduction of \cite{GV19Pn} and the references therein. 
In dimension one, R. Benedetto \cite{Benedetto05} proved that if $f$ is a non-isotrivial polynomial over the function field $K$ of transcendence degree one and of any characrestic, then there are only finitely many preperiodic $K$-points, together with the gap property of the canonical height function. Then M. Baker \cite{Baker09} generalized it to all non-isotrivial rational functions and all function fields with product formulas. L. Demarco \cite{Demarco16} improved the condition to non-isotrivial over $K$ using arguments from complex dynamics, where $K$ is a complex function field of transcendence degree one (see also related works by R. Dujardin and C. Favre \cite{DujardinFavre08} and C. McMullen \cite{McMullen87}). 

In higher dimension, T. Gauthier and G. Vigny \cite{GV19Pn} proved geometric Northcott properties of canonical height zero (i.e. without gap) for all polarized endomorphisms of projective varieties. They also interpret the canonical height function in terms of a bifurcation current. Note that Z. Chatzidakis and E. Hrushovski \cite{chatzidakis_hrushovski_2008} proved a model-theoretic version of Corollary \ref{corollary} for $\P^k$.

\medskip

\textbf{Acknowledgements}
I would like to thank my advisor T. Gauthier for numerous helpful discussions and all the time he spent with me. I would like to thank B. Dang, C. Favre and G. Vigny for useful discussions. I also thank S. Cantat and the anonymous referee for their useful comments and suggestions.

\section{Canonical height function}
\label{section height function}

In this section we recall the construction of the canonical height function $\hat{h}_f$ associated with $f.$ Canonical height functions were introduced by G. Call and J. Silverman \cite{CallSilverman93} in the global height fields case. The interpretation by means of intersection numbers can be found for example in \cite{Gubler2006TheBC}, see also \cite{CGHX}. We use the additive notation for tensor product of line bundles.

Let us first construct by induction a sequence of models ${(\pi_n:\X_n\to \B,\f_n,\L_n)}_{n\in \N}$ as follows. 
Let $\X=\X_0 = \mathbb{P}^n_{\C} \times \B,$ with $\pi=\pi_0: \X \to \B$ the second projection. Fix an ample line bundle $\M$ of $\B$ such that $\L=\L_0:=\O_{\mathbb{P}^n_{\C}}(1)+\pi^{*}(\M)$ is an ample line bundle on $\X$. There exists an open subset $\Lambda$ of $\B$ such that $f$ extends to a fibered endomorphism $\f$ on $\X_{\Lambda}:=\mathbb{P}^n_{\C} \times \Lambda$ and $\f^{*}\L|_{\Lambda}\cong d\L|_{\Lambda}.$ We call $\X_\Lambda$ the regular part of the model $\X$. For any parameter $t\in \Lambda$, the fiber $\X_t:=\pi^{-1}(t)$ will be called a regular fiber.

Suppose now $(\pi_n:\X_n\to \B,\f_n,\L_n)$ constructed. Let $\X_{n+1}$ be the normalization of the Zariski closure of the graph of $(\mathrm{id},\f_n): \X_{n,\Lambda}\to \X_n \times \X_n$ preserving $\X_{\Lambda}$. Let $\phi_{n+1}:\X_{n+1} \to \X_n$ and $g_{n+1}: \X_{n+1} \to \X_n$ be the first and the second projection. Let $\pi_{n+1}:=\pi_{n}\circ\phi_{n+1}, \f_{n+1}:=\phi_{n+1,\Lambda}^{-1} \circ \f_n \circ g_{n+1,\Lambda}$ and $\L_{n+1}:=\frac{1}{d}g_{n+1}^{*}\L_n, $ as the following diagram shows.
\[
\begin{tikzcd}[]
\mathfrak{X}_{n+1} \arrow[r, "\mathfrak{f}_{n+1}", dashed] \arrow[d, "g_{n+1}"'] & \mathfrak{X}_{n+1} \arrow[d, "\phi_{n+1}"] \\
\mathfrak{X}_n \arrow[r, "\mathfrak{f}_n"', dashed]                              & \mathfrak{X}_n                            
\end{tikzcd}
\]
We next construct the canonical height function associated with this sequence of models. Let $x\in X(\overline{K})$ be a point. Denote by $k(x) $ its residue field and $\sigma_{x,n}$ its Zariski closure in the model $\X_n.$ Define the canonical height function associated with $f$ as the following limit of intersection number.
$$\hat{h}_{f}(x):=\frac{1}{[k(x):K]} \lim_{n}  \sigma_{x,n} \cdot c_1(\L_n) \cdot c_1(\pi_n^{*}\M)^{b-1}.$$
It is well-known that the limit exists, see Proposition \ref{higehtwelldefined} below. We recall the proof since we will need it and also for the convenience of the reader. For more properties, see for example Theorem 11.18 of \cite{Gubler2003ASNSP} and Theorem 3.6 of \cite{Gubler2006TheBC}.
\begin{prop}
\label{higehtwelldefined}
The canonical height function is well defined, i.e. the limit $\lim_{n}  \sigma_{x,n} \cdot c_1(\L_n)\cdot c_1(\pi^{*}\M)^{b-1}$ exists. 

It satisfies the following two properties:
\begin{itemize}
    \item $\hat{h}_{f}(x) = \sigma_{x} \cdot c_1(\L)\cdot c_1(\pi^{*}\M)^{b-1} + O(1);$
    \item $\hat{h}_{f}(f(x))=d\hat{h}_{\f}(x).$
\end{itemize}
The bound $O(1)$ is independent of $x$.
\end{prop}
\begin{proof}
For all $n\in \N^{*}$, we have
\begin{align*}
    \phi_{n+1}^* \L_n - \L_{n+1} &=\phi_{n+1}^*(\frac{1}{d}{g_{n}}^{*} \L_{n-1}) - \frac{1}{d}{g_{n+1}}^{*} \L_{n}\\
    &=\frac{1}{d}g_{n+1}^{*}(\phi_n^*\L_{n-1} - \L_n )\\
    &=\frac{1}{d^{n}} g_{n+1}^{*}\cdots g_2^{*}(\phi_1^*\L_0 - \L_1),
\end{align*}

Since ${\phi_1^{*}(\L_0)}_{\eta} \cong \L_{1,\eta},$ where $\eta$ is the generic point of $\B$, there exists a vertical divisor $\V$ on $\X_1$ such that $\phi_1^{*}(\L_0) - \L_1 \cong \O_{\X_1}(\mathscr{V})$. 
Hence there exists a effective divisor $\mathscr{N}$ on $\B$ such that $-\pi_1^{*}(\mathscr{N}) < \mathscr{V} < \pi_1^{*}(\mathscr{N}) $.

The following computation implies that we have a geometric series, whence the convergence of the intersection numbers and the first property:
\begin{align*}
    |\sigma_{x,n}\cdot c_1(\L_n) \cdot c_1(\pi_n^{*}\M)^{b-1} &- \sigma_{x,n+1}\cdot c_1(\L_{n+1}) \cdot c_1(\pi_{n+1}^{*}\M)^{b-1}|\\
    &= |\sigma_{x,n+1}\cdot c_1(\phi_{n+1}^* \L_n - \L_{n+1})\cdot c_1(\pi_{n+1}^{*}\M)^{b-1}|\\
    &= |\frac{1}{d^n}\sigma_{x,n+1} \cdot c_1(g_{n+1}^{*}\cdots g_2^{*}(\phi_1^*\L_0 - \L_1)) \cdot c_1(\pi_{n+1}^{*}\M)^{b-1}|\\
    &< |\frac{1}{d^n}\sigma_{x,n+1} \cdot c_1( \pi_{n+1}^{*}\O_\B(\mathscr{N}))\cdot c_1(\pi_{n+1}^{*}\M)^{b-1}|\\
    &= \frac{[k(x):K]}{d^n} \mathscr{N} \cdot \M^{b-1}.
\end{align*}
Note that the bound at the end divided by $[k(x):K]$ is independent of $x$. The second property comes from the projection formula: $g_{n+1}(\sigma_{x,n+1})\cdot c_1(\L_n) = \sigma_{x,n+1} \cdot c_1(g_{n+1}^{*}(\L_n))$.
\end{proof}
\begin{rem}\normalfont
\begin{enumerate}
    \item The two properties uniquely determine the canonical height function.
    \item The canonical height function does not depend on the chosen sequence of models. It depends on the polarization $(\B,\M)$ but the vanishing of $\hat{h}_f$ is independent of $(\B,\M)$.
\end{enumerate}
\end{rem}

The next proposition gives a characterization of the points of canonical height zero. The notation $f^n(\sigma_x)$ means the Zariski closure of the point $f^n(x)\in X(\overline{K})$ in the model $\X.$
\begin{prop}
\label{c_u}
    The canonical height function is non-negative, and the following assertions are equivalent:
    \begin{enumerate}
        \item $\hat{h}_f(x)=0$;
        \item There exists a positive constant $c_u$ such that for all $x\in X(\overline{K})$ and $n\in \N,$
        $$\frac{1}{[k(x):K]}f^n(\sigma_x) \cdot c_1(\L) \cdot c_1(\pi^{*}\M)^{b-1}\leq c_u;$$
        \item There exists a positive constant $c'_u$ such that for all $x\in X(\overline{K})$ and $n\in \N,$
        $$\frac{1}{[k(x):K]}f^n(\sigma_x) \cdot c_1(\L)^b \leq c'_u.$$
    \end{enumerate}
\end{prop}

\begin{proof}
    The non-negativity is clear since the involved line bundles $\L_n$ and $\pi_{n}^{*}\M$ are nef.
    It is also clear that the uniform boundedness implies vanishing of the canonical height.
    
    Suppose now that $x\in X(\Bar{K})$ has canonical height zero. We can assume that $X\in X(K).$ By the preceding proposition there exists a uniform positive constant $c_u>0$ such that 
    $$|\hat{h}_f(x) - \sigma_x \cdot c_1(\L) \cdot c_1(\pi^{*}\M)^{b-1} | < c_u.$$ 
    Since $\hat{h}_f(f^n(x))= d^n \hat{h}_f (x)=0$, we have that 
    $$ f^n(\sigma_x) \cdot c_1(\L) \cdot c_1(\pi^{*}\M)^{b-1}< c_u,$$ 
    which shows the equivalence between (1) and (2).

    Since $\L$ is ample, (3) implies (2). It thus remains to show that (2) implies (3). We can assume that $x\in X(K).$ The case $b=1$ is trivial, thus we may suppose that $b\geq 2$. Using Theorem 3 of \cite{DangLondon} (see also \cite{Popovici15} and \cite{Xiao15}), there exists a constant $c>0$, depending only on $b,$ such that, for all integers $j=0,\cdots,b-2$
    \begin{align*}
        \sigma_x\cdot \mathrm{c}_1(\L)^{b-j} &\cdot \mathrm{c}_1(\pi^{*}(\M))^{j}\\
        & \leq c \frac{\sigma_x \cdot \mathrm{c}_1(\L)^{b-j-1} \cdot \mathrm{c}_1(\pi^{*}(\M))^{j+1}}{\sigma_x \cdot \mathrm{c}_1(\pi^{*}(\M))^{b}} \sigma_x \cdot  \mathrm{c}_1(\L) \cdot \mathrm{c}_1(\pi^{*}(\M))^{b-1}\\
         & \leq c' \sigma_x \cdot \mathrm{c}_1(\L)^{b-j-1} \cdot \mathrm{c}_1(\pi^{*}(\M))^{j+1},
    \end{align*}
where $c':= c \sigma_x \cdot  \mathrm{c}_1(\L) \cdot \mathrm{c}_1(\pi^{*}(\M))^{b-1} / \sigma_x \cdot \mathrm{c}_1(\pi^{*}(\M))^{b} \leq c c_u/\deg(\M)$.
    
Thus by induction, we get 
\begin{align*}
    \sigma_x\cdot \mathrm{c}_1(\L)^{b} \leq {(c')}^{b-1} \sigma_x \cdot \mathrm{c}_1(\L) \cdot \mathrm{c}_1(\pi^{*}(\M))^{b-1}.
\end{align*}
It suffices to set $c'_u:={(c')}^{b-1}$.
\end{proof}

\section{Preliminaries for some parameter spaces}
Remark that $K-$rational points of $X$ can be seen as rational sections of $\pi : \X \to \B$. There are natural schemes/complex spaces parameterizing these objects which will be used in the proof of our main theorem. In this section we define these parameter spaces and give some basic properties. References are also given.

Recall first the notion of the Hilbert scheme constructed by Grothendieck (\cite{GroHilbert}). Let $\X$ be a complex projective variety. The Hilbert scheme $\mathrm{Hilb}(\X)$ is a locally Noetherian complex scheme which represents the functor $\underline{\mathrm{Hilb}}_\X$ defined as follows: to a complex scheme $T$ we associate the set of closed subschemes of $\X\times T$ which is flat over $T$. By definition, there is a flat family of closed subschemes $p_\X: \Gamma(\mathrm{Hilb}(\X)) \subset \X \times \mathrm{Hilb}(\X) \to \mathrm{Hilb}(\X)$ such that the geometric fibers are in one-to-one correspondence with the closed subschemes of $\X.$ We call $\Gamma(\mathrm{Hilb}(\X))$ the graph of $\mathrm{Hilb}(\X)$. Denote by $\mathrm{Hilb}_n(\X)$ the open subset of $\mathrm{Hilb}(\X)$ consisting of subschemes of pure dimension $n.$ Given a surjective morphism of complex projective varieties $\pi : \X \to \B$, denote by $\mathrm{Rat}(\B,\X)$ the open subset of the Hilbert scheme $\mathrm{Hilb}(\B\times \X)$ whose $\C-$points correspond to graphs in $\B\times \X$ of rational sections of $\pi$ (see the proof of Proposition 1.7 in \cite{Masaki}).

We now recall the definition of cycle space (Chapitre IV of \cite{barlet2014cycles}, see also Chapter I of \cite{kollar1999rational}). Consider the set $\CC_n(\X)$ of $n-$cycles in a projective variety $\X$ of dimension $m,$ that is, all finite sums $C=\sum_i m_iC_i$ of irreducible varieties $C_i$ of dimension $n$ with positive coefficients $m_i$. We give $\CC_n(\X)$ the topology such that the natural map from $\CC_n(\X)$ to the space of positive currents is a closed immersion. 

We can endow $\CC_n(\X)$ with a complex structure such that it is actually the analytic complex space associated with the Chow scheme $\mathrm{Chow}_n(\X)$ of $\X$ which parameterizes algebraic n-cycles of $\X$. We will write  $\CC_n(\X)$ and $\mathrm{Chow}_n(\X)$ indifferently. There is a morphism (called Douady-Barlet morphism in the analytic case, and Hilbert-Chow in the algebraic case) $$\Theta_\X^n: \mathrm{Hilb}_n(\X)_{\mathrm{red}} \to \CC_n(\X)$$ defined as follows (where 'red' means the reduced induced scheme structure). Let $C$ be a closed subscheme and $C_1,\cdots,C_t$ the irreducible components. Denote by $m_i:=l_{\O_{C_i,C}}(\O_{C_i,C})$ the length of $\O_{C_i,C}$, then $\Theta_\X^n(C):=\sum m_i[C_i]$ (see \cite{fulton2012intersection}).

Moreover, if $f :\X \to \mathfrak{Y}$ is a morphism of projective varieties, there is an induced morphism $$f_{*}:\CC_n(\X) \to \CC_n(\mathfrak{Y})$$ 
defined as follows. First let $C\in \CC(\X)$ be an irreducible variety of dimension $n$. If $f(C)$ has dimension strictly less than $n$, then we set $f_{*}(C)$ to be the empty $n-$cycle. Otherwise, the restriction $f|_{C}:C \to f(C)$ is generically finite of degree $m$ and we define $f_{*}(C):=mf(C)$. We then extend  the map $f_{*}$ by linearity to all cycles (see Chapitre X of \cite{barletCyclesII}).

\section{Proof of Theorem \ref{theorem} and Corollary \ref{corollary}}

\begin{proof}[Proof of Theorem \ref{theorem}]
    We first construct an appropriate parameter space using the preliminaries of the previous section.
    Fix a positive integer $k$, Let $\X^{'}_k$ be the normalization of the Zariski closure of the graph 
    $$(id,f^k) :\X_\Lambda \to \X \times \X$$ 
    and denote by $\pi_i^k, i=1,2$ the two projections from $\X^{'}_k$ to $\X$. Set $\L_k:=\sum_{i=1}^2 {(\pi_i^k)}^{*}\L$, it is an ample divisor on $\X^{'}_k$. The two projections induce morphisms at the level of cycle spaces that still denoted by $\pi_i^k$. Let $\CC_{c'_u}(\X, \L)$ be the (possibly reducible) projective variety formed of $b$-cycles $C$ such that $C\cdot \mathrm{c}_1(\L)^b \leq c'_u.$ Consider the subvariety $$S_k:=\pi_1^k(\cap_{i=1}^2 {(\pi_i^k)}^{-1}(\CC_{c'_u}(\X, \L)) )$$ 
    in $\CC_{c'_u}(\X, \L)$ and take the intersection 
    $$S:=\bigcap_k S_k$$ 
    over all $k\in \N$, then $S$ consists of $b$-cycles with stable horizontal part (i.e. Zariski closure of stable points), and whose vertical part (i.e. its projection to $\B$ is a proper subset) is supported over $\B \setminus \Lambda.$

    Since we only want to consider stable $K-$points, we continue to modify our variety $S$. Denote by $q:\B \times \X \to \X$ the second projection and by 
    $$\Theta: \mathrm{Rat}(\B,\X) \to \CC_b(\X)$$ 
    the composition $\Theta:=q_{*} \circ \Theta_{\B\times \X}^b$ restricted on $\mathrm{Rat}(\B,\X)$. Then the preimage $Z:=\Theta^{-1}(S)$ of $S$ by $\Theta$ consists of only stable irreducible horizontal $b-$cycles which are birational to $\B$.
    
    \medskip
    
    Let us prove the first point (1). By the Noetherianity of the Zariski topology, the intersection $S=\cap_k S_k$ is stable for a large but finite intersections, say the first $N$ factors $S_k$. Whence the gap for the canonical height function around zero. Indeed, by Proposition \ref{higehtwelldefined}, we have 
\begin{align*}
    |\hat{h}_{f}(x) - \sigma_{x} \cdot c_1(\L)\cdot c_1(\pi^{*}\M)^{b-1}|<c_u.
\end{align*}
Since the bound is independent of points, we get
\begin{align*}
    f^n(\sigma_x)\cdot c_1(\L)\cdot c_1(\pi^{*}\M)^{b-1}<c_u + d^n\hat{h}_{f}(x),
\end{align*}
for all $n\in \N$.
If $\hat{h}_f(x)<c_u/d^N$, then $f^n(\sigma_x)\cdot c_1(\L)\cdot c_1(\pi^{*}\M)^{b-1}<2c_u$ for all $n=1,\cdots,N$, hence for all $n\in \N$ by Noetherianity, which implies $\hat{h}_f(x)=0$ by definition.

\medskip

We now prove the second point (2). Suppose that the set of stable $K-$points are Zariski dense. Take an irreducible component of maximal dimension of $Z$, by abuse of notation, we still write it $Z$. Denote by $\Gamma(Z)$ the graph of $Z$ in $\X \times Z$ and $\pi_1$, $\pi_2$ the two projections of $\Gamma(Z)$ to $\X$ and $Z$ respectively. 

Let us prove that $\pi_1$ is a birational map. Let $p\in \X_{t_0}, t_0\in \Lambda$ be a repelling periodic point of $f_{t_0}$ in a regular fiber. Up to taking an iteration of $f$, we can suppose it is fixed. By the implicit function theorem, there is a local analytic continuation of $p$ as repelling fixed point: there exists an analytic open subset $U\subset \B$ containing $t_0$ and a local section $\sigma_0 :U \to \X_U$ such that for all $t \in U,$ $\sigma_0(t)$ is a repelling fixed point of $f_t$ in $\X_t$. Let $\sigma_x$ be the Zariski closure of a $K-$point $x\in X$ passing through the point $p.$ There is an open neighborhood $\U$ of $p$ so that, there exists a real number $\mu > 1$ such that, in every fiber $\U_t,$ where $t\in U,$ the induced morphism $f_t$ satisfies 
$$\mathrm{dist}(f_t(x),\sigma_0(t)) > \mu\cdot \mathrm{dist}(x,\sigma_0(t)),$$
for all $x\in \U_t$. Since the $b-$cycles $f^n(\sigma_x)$ are stable for all $n\in \N$, we can extract a subsequence converging to a $b-$cycle $C_0$. It cannot have a vertical component in the regular part $\X_\Lambda$, hence $\sigma_x$ has to coincide with $\sigma_0$. Thus if an irreducible horizontal $b-$cycle in $Z$ passes through a repelling periodic point, it is unique. Since repelling periodic points are Zariski dense (see \cite{dnt,BriendDuval99}), $\pi_{1,\Lambda}$ is generically finite of degree 1, thus it is birational.

Let $\Y\subset \X_\Lambda$ be a Zariski open subset over which $\pi_1$ is an isomorphism. For all $t\in \Lambda,$ $\Psi_t:= \pi_2 \circ \pi_1^{-1}|_{\Y_t}: \Y_t \to Z$ is a birational map. Hence for all $t,t'\in \Lambda $, the endomorphisms $f_t$ and $f_{t'}$ are birationally conjugated by $\Psi_{t}^{-1} \circ \Psi_{t'}$.
\end{proof}
~\\ 

\begin{proof}[Proof of Corollary \ref{corollary}] 
We keep the notation of the previous proof. Since $f$ is not isotrivial, the dimension dim($Z$) of $Z$ (which is irreducible by assumption) is zero or one. Denote by $\ZZ$ the image of $\pi_1:\Gamma(Z)\to \X.$

If dim($Z$)=1, then the generic fiber of $\ZZ$ in $X$ is a singular curve with infinitely many $K$-points, hence 
it is birational to a projective line or to an elliptic curve. There are finitely many such curves and they are preperiodic by $f$. Denote by $E_1$ the union of these curves. Let $C$ be any such curve in $E_1$. There exist positive integers $n,m$ such that $f^n(C)$ is fixed by $f^m$. If $f^n(C)$ is of genus zero, then by the main theorem, $f^m: f^n(C) \to f^n(C)$ is isotrivial. If $f^n(C)$ is of genus one, we can assume it is normal. The same argument as in the proof of the main theorem implies that general fibers of a model $\mathfrak{C}$ of $f^n(C)$ are isomorphic, where by a model we mean a surjective morphism $\pi_C : \mathfrak{C} \to \B$ whose generic fiber is $f^n(C)$. Thus $f^n(C)$ is isotrivial (see e.g. \cite{Sandor}). 

If dim($Z$)= 0, either the generic point of $\ZZ$ is a preperiodic $K-$point in $X,$ or it will be sent to $E_1$ by a iteration of $f$, since it is a stable point. Let $E_0$ be the union of these points which are sent to $E_1$ by a iteration, then set $E=E_0\cup E_1$.
\end{proof}

\bibliographystyle{plain}
\bibliography{mybio}
\end{document}